\documentclass[11pt,letterpaper]{amsart}
\usepackage[utf8]{inputenc}
\usepackage{setspace}
\usepackage{mathrsfs}
\usepackage{amssymb}
\usepackage{latexsym}
\usepackage{amsfonts}
\usepackage{amsmath}
\usepackage{eucal}
\usepackage{bm}
\usepackage{bbm}
\usepackage{graphicx}
\usepackage[english]{varioref}
\usepackage[nice]{nicefrac}
\usepackage[all]{xy}
\usepackage{amsthm}
\usepackage{amssymb,amsthm,upref,amscd}

\def\op{\operatorname}

\def\mmod{\kern-1pt\operatorname{-mod}}

\newtheorem{theorem}{Theorem}[section]
\newtheorem{lemma}[theorem]{Lemma}
\newtheorem{definition}[theorem]{Definition}

\newtheorem{remark}[theorem]{Remark}
\newtheorem{proposition}[theorem]{Proposition}
\newtheorem{corollary}[theorem]{Corollary}

\numberwithin{equation}{section}

\begin{document}

\title[Complex representations of $SL_2(\bar{\mathbb{F}}_q)$]{Certain complex representations of $SL_2(\bar{\mathbb{F}}_q)$ }

\author{Junbin Dong}
\address{Institute of Mathematical Sciences, ShanghaiTech University, 393 Middle Huaxia Road, Pudong, Shanghai 201210, China.}
\email{dongjunbin@shanghaitech.edu.cn}


\date{September 20, 2022}


\begin{abstract}
We introduce the representation category $\mathscr{C}({\bf G})$ for a connected reductive algebraic  group ${\bf G}$ which is defined over a finite field $\mathbb{F}_q$ of $q$ elements. We show that this category has many good properties for ${\bf G}=SL_2(\bar{\mathbb{F}}_q)$. In particular, it is an abelian category and a highest weight category. Moreover, we classify the simple objects in $\mathscr{C}({\bf G})$ for ${\bf G}=SL_2(\bar{\mathbb{F}}_q)$.
\end{abstract}

\maketitle

\section{Introduction}

Let ${\bf G}$ be a connected reductive algebraic  group defined over the finite field $\mathbb{F}_q$ of $q$ elements. Let  $\Bbbk$ be another field and all representations in this paper are over $\Bbbk$. According to  a result of Borel and Tits \cite[Theorem 10.3 and Corollary 10.4]{BT}, we know that except the trivial representation, all other irreducible representations of $\Bbbk {\bf G}$ (the group algebra of $\bf G$) are infinite-dimensional if ${\bf G}$ is  semisimple and $\Bbbk $ is infinite with $\op{char}\Bbbk\neq \op{char} \bar{\mathbb{F}}_q$.  Denote by $G_{q^a}$ the set of $\mathbb{F}_{q^a}$-points of $\bf G$, then we have ${\bf G}=\bigcup G_{q^a}$.  With this basic fact, N.H. Xi studied the abstract representations of ${\bf G}$ over ${\bf \Bbbk}$ by taking the direct limit of the finite-dimensional representations of $G_{q^a}$ and he got many interesting results in \cite{Xi}. In particular, he showed that  the infinite-dimensional Steinberg module is irreducible when
$\op{char}\Bbbk=0$ or $\op{char}\Bbbk= \op{char} \bar{\mathbb{F}}_q$. Afterwards, R.T. Yang proved the irreducibility of the Steinberg module for any field $\Bbbk$ with $\op{char}\Bbbk\ne \op{char} \bar{\mathbb{F}}_q$ (see \cite{Yang}). Later, motivated by Xi's idea, the structure of the permutation module $\Bbbk [{\bf G}/{\bf B}]$ (${\bf B}$ is a fixed Borel subgroup of ${\bf G}$ ) was studied in  \cite{CD1} for the cross characteristic case and in \cite{CD2} for the defining characteristic case.
We studied the general abstract induced module $\mathbb{M}(\theta)=\Bbbk{\bf G}\otimes_{\Bbbk{\bf B}}{\Bbbk}_\theta$ in \cite{CD3} for any field $\Bbbk$ with $\op{char}\Bbbk\neq \op{char} \bar{\mathbb{F}}_q$ or $\Bbbk=\bar{\mathbb{F}}_q$, where $\bf T$ is a maximal  torus contained in a Borel subgroup $\bf B$ and $\theta$  is  a character of ${\bf T}$ which can also be regarded as a character of ${\bf B}$ through the homomorphism ${\bf B}\rightarrow{\bf T}$. The  induced module $\mathbb{M}(\theta)$  has a composition series (of finite length) if $\op{char}\Bbbk\neq \op{char} \bar{\mathbb{F}}_q$. In the case $\Bbbk=\bar{\mathbb{F}}_q$ and $\theta$ is a rational character, $\mathbb{M}(\theta)$ has such composition series if and only if $\theta$ is antidominant (see \cite{CD3} for details). In both cases, the composition factors of $\mathbb{M}(\theta)$  are $E(\theta)_J$ with $J\subset I(\theta)$ (see Section 2 for the explicit setting).

Now we have a large class of irreducible $\Bbbk {\bf G}$-modules. Let $\Bbbk{\bf G}$-Mod be the $\Bbbk {\bf G}$-module category. However, this category is too big and thus in the paper \cite{D1}, we introduce the principal representation category $\mathscr{O}({\bf G})$ which was supposed to have many good properties.
In particular, we conjectured that this category  is a highest weight category in the sense of Cline, Parshall and Scott \cite{CPS}. However, recently X.Y.Chen constructed a counter example (see \cite{Chen}) to show that this conjecture is not valid in general with the setting given in \cite[Section 4]{D1}. Thus it deserves to explore other categories besides the category $\mathscr{O}({\bf G})$  in \cite{D1}. We hope that there is a category which satisfies certain good properties such as ``finite-ness" and ``semi-simplicity", which is also like the BGG category $\mathscr{O}$ in the representations of complex semisimple Lie algebras.  In this paper we introduce a full subcategory $\mathscr{C}(\bf G)$  of $\Bbbk{\bf G}$-Mod, whose objects are finitely generated by some ${\bf T}$-eigenvectors (see Section 2 for the  definition of $\mathscr{C}(\bf G)$).   The main part of this paper is  devoted to study the category $\mathscr{C}(\bf G)$ for ${\bf G}=SL_2(\bar{\mathbb{F}}_q)$.

The rest of this paper is organized as follows:  Section 2 contains some preliminary results and  we also introduce the category $\mathscr{C}({\bf G})$ in this section. From Section 3 to Section 5, we assume that $\Bbbk$ is an algebraically closed field of characteristic $0$ and study the category $\mathscr{C}({\bf G})$ for ${\bf G}=SL_2(\bar{\mathbb{F}}_q)$. In Section 3, we classify the simple $\Bbbk {\bf G}$-modules with $\bf T$-stable lines. In particular, we get all the simple objects in $\mathscr{C}({\bf G})$. Then we show that $\mathscr{C}({\bf G})$ is an abelian category and has certain good properties in Section 4. In Section 5 we prove that $\mathscr{C}({\bf G})$ is a highest weight category.

\section{Background and preliminary results}

Let ${\bf G}$ be a connected reductive algebraic group over  $\bar{\mathbb{F}}_q$, the algebraic closure of  $\mathbb{F}_q$, e.g. ${\bf G}=GL_n(\bar{\mathbb{F}}_q)$, $SL_n(\bar{\mathbb{F}}_q)$.  Let ${\bf B}$ be a Borel subgroup, and ${\bf T}$ be a  maximal torus contained in ${\bf B}$, and ${\bf U}=R_u({\bf B})$ be the  unipotent radical of ${\bf B}$. We identify ${\bf G}$ with ${\bf G}(\bar{\mathbb{F}}_q)$ and do likewise for the various subgroups of ${\bf G}$ such as ${\bf B}, {\bf T}, {\bf U}$ $\cdots$. We denote by $\Phi=\Phi({\bf G};{\bf T})$ the corresponding root system, and by $\Phi^+$ (resp. $\Phi^-$) the set of positive (resp. negative) roots determined by ${\bf B}$. Let $W=N_{\bf G}({\bf T})/{\bf T}$ be the corresponding Weyl group. We denote by $\Delta=\{\alpha_i\mid i\in I\}$ the set of simple roots and by $S=\{s_i:=s_{\alpha_i}\mid i\in I\}$ the corresponding simple reflections in $W$. For each $\alpha\in\Phi$, let ${\bf U}_\alpha$ be the root subgroup corresponding to $\alpha$ and we fix an isomorphism $\varepsilon_\alpha: \bar{\mathbb{F}}_q\rightarrow{\bf U}_\alpha$ such that $t\varepsilon_\alpha(c)t^{-1}=\varepsilon_\alpha(\alpha(t)c)$ for any $t\in{\bf T}$ and $c\in\bar{\mathbb{F}}_q$. For any $w\in W$, let ${\bf U}_w$ (resp. ${\bf U}_w'$) be the subgroup of ${\bf U}$ generated by all ${\bf U}_\alpha$  with $w(\alpha)\in\Phi^-$ (resp. $w(\alpha)\in\Phi^+$). One is refereed \cite{Car} for more details.

Now let $\Bbbk$ be an algebraically closed field of characteristic $0$ and all the representations in this paper are over $\Bbbk$. Let $\widehat{\bf T}$ be the set of characters of ${\bf T}$. Each $\theta\in\widehat{\bf T}$ can be regarded as a character of ${\bf B}$ by the homomorphism ${\bf B}\rightarrow{\bf T}$. Let ${\Bbbk}_\theta$ be the corresponding ${\bf B}$-module. We consider the induced module $\mathbb{M}(\theta)=\Bbbk{\bf G}\otimes_{\Bbbk{\bf B}}{\Bbbk}_\theta$. Let ${\bf 1}_{\theta}$ be a fixed nonzero element in ${\Bbbk}_\theta$. We abbreviate $x{\bf 1}_{\theta}:=x\otimes{\bf 1}_{\theta}\in\mathbb{M}(\theta)$ for $x\in {\bf G}$. It is not difficult to see that $\mathbb{M}(\theta)$ has a basis $\{u \dot{w} {\bf 1}_{\theta}\mid w\in W,  u\in {\bf U}_{w^{-1}}\}$ by the Bruhat decomposition, where $\dot{w}$ is a fixed representative of $w \in W$.

For each $i \in I$, let ${\bf G}_i$ be the subgroup of $\bf G$ generated by ${\bf U}_{\alpha_i}, {\bf U}_{-\alpha_i}$ and set ${\bf T}_i= {\bf T}\cap {\bf G}_i$.
For $\theta\in\widehat{\bf T}$, define the subset $I(\theta)$ of $I$ by $$I(\theta)=\{i\in I \mid \theta| _{{\bf T}_i} \ \text {is trivial}\}.$$
The Weyl group $W$ acts naturally on $\widehat{\bf T}$ by
$$(w\cdot \theta ) (t):=\theta^w(t)=\theta(\dot{w}^{-1}t\dot{w})$$
for any $\theta\in\widehat{\bf T}$.

Let $J\subset I(\theta)$, and ${\bf G}_J$ be the subgroup of $\bf G$ generated by ${\bf G}_i$, $i\in J$. We choose a representative $\dot{w}\in {\bf G}_J$ for each $w\in W_J$ (the standard parabolic subgroup of $W$). Thus, the element $w{\bf 1}_\theta:=\dot{w}{\bf 1}_\theta$  $(w\in W_J)$ is well-defined. For $J\subset I(\theta)$, we set
$$\eta(\theta)_J=\sum_{w\in W_J}(-1)^{\ell(w)}w{\bf 1}_{\theta},$$
where $\ell$ is the length function on $W$.  Let $\mathbb{M}(\theta)_J=\Bbbk{\bf G}\eta(\theta)_J$ the $\Bbbk {\bf G}$-module which is generated by $\eta(\theta)_J$.

For $w\in W$, denote by  $\mathscr{R}(w)=\{i\in I\mid ws_i<w\}$.  For any subset $J\subset I$, we set
$$
\aligned
X_J &\ =\{x\in W\mid x~\op{has~minimal~length~in}~xW_J\}.
\endaligned
$$
We have the following proposition.

\begin{proposition}\cite[Proposition 2.5]{CD3} \label{MJ=KUW} \normalfont
For any $J\subset I(\theta)$, the $\Bbbk {\bf G}$-module $\mathbb{M}(\theta)_J$ has the form
\begin{align} \label{MJ}
\mathbb{M}(\theta)_J=\sum_{w\in X_J}\Bbbk{\bf U}\dot{w}\eta(\theta)_J=\sum_{w\in X_J}\Bbbk{\bf U}_{w_Jw^{-1}}\dot{w}\eta(\theta)_J.
\end{align}
In particular, the set $\{u\dot{w}\eta(\theta)_J \mid w\in X_J, u\in {\bf U}_{w_Jw^{-1}} \}$ forms a basis of $\mathbb{M}(\theta)_J$.
\end{proposition}

For $J\subset I(\theta)$, define
$$E(\theta)_J=\mathbb{M}(\theta)_J/\mathbb{M}(\theta)_J',$$
where $\mathbb{M}(\theta)_J'$ is the sum of all $\mathbb{M}(\theta)_K$ with $J\subsetneq K\subset I(\theta)$. We denote by $C(\theta)_J$ the image of $\eta(\theta)_J$ in $E(\theta)_J$.
We also set
$$
\aligned
Z_J &\ =\{w\in X_J \mid \mathscr{R}(ww_J)\subset J\cup (I\backslash I(\theta))\}.
\endaligned
$$
The following proposition gives a basis of $E(\theta)_J$.

\begin{proposition} \cite[Proposition 2.7]{CD3} \label{DesEJ} \normalfont
For $J\subset I(\theta)$, we have
\begin{align} \label{EJ}
E(\theta)_J=\sum_{w\in Z_J}\Bbbk {\bf U}_{w_Jw^{-1}}\dot{w}C(\theta)_J.
\end{align}
In particular, the set $\{u\dot{w}C(\theta)_J \mid w\in Z_J, u\in {\bf U}_{w_Jw^{-1}} \}$ forms a basis of $E(\theta)_J$.
\end{proposition}

The $\Bbbk {\bf G}$-modules $E(\theta)_J$ are  irreducible and thus we get all the composition factors of $\mathbb{M}(\theta)$ (see  \cite[Theorem 3.1]{CD3}). According to \cite[Proposition 2.8]{CD3}, one has that  $E(\theta_1)_{K_1}$ is isomorphic to $E(\theta_2)_{K_2}$ as $\Bbbk {\bf G}$-modules if and only if $\theta_1=\theta_2$ and $K_1=K_2$.

For a $\Bbbk {\bf G}$-module $M$, an element $\xi \in M$ is called a ${\bf T}$-eigenvector if $ t\xi= \lambda(t) \xi$ for some $\lambda \in \widehat{\bf T}$. Set $M_{\lambda}= \{\xi \in M \mid t\xi= \lambda(t) \xi \}$, which is called the weight space corresponding to $\lambda \in \widehat{\bf T}$. A character $\lambda \in \widehat{\bf T}$ is called a weight of $M$ if $M_{\lambda}\ne 0$ and then we denote the  weight set of $M$ by $\text{Wt(M)}$.
Let $M^{\bf T}= \displaystyle \bigoplus_{\lambda \in \widehat{\bf T}} M_{\lambda} $ and set  $\mathfrak{d}(M) =  \dim M^{\bf T}$. With previous discussion and the form of $\mathbb{M}(\theta)_J$ and $E(\theta)_J$ (see (\ref{MJ}) and (\ref{EJ})), we are interested in the  $\Bbbk {\bf G}$-module $M$ which satisfies the following condition:

\noindent ($\star$) $n=\mathfrak{d}(M)< +\infty$ and  there exist ${\bf T}$-eigenvectors $\xi_1, \xi_2, \dots, \xi_n$ such that $M\cong \displaystyle  \bigoplus_{i=1}^n \Bbbk {\bf U} \xi_i$ as $\Bbbk {\bf B}$-modules.

Let $\mathscr{C}(\bf G)$ be the full  subcategory of $\Bbbk {\bf G}$-Mod, which consists of the $\Bbbk {\bf G}$-modules satisfying the condition ($\star$). From the definition of $\mathscr{C}(\bf G)$, it seems very difficult to judge whether it is an abelian category. Naturally, we have the following fundamental questions: (1) Is the category $\mathscr{C}(\bf G)$ an abelian category?
(2) Is this category noetherian or artinian? (3) Classify all the simple objects in $\mathscr{C}(\bf G)$. We will solve these problems for ${\bf G}= SL_2(\bar{\mathbb{F}}_q)$ in the following discussion.

\section{Simple modules with $\bf T$-stable lines.}

From now on, let ${\bf G} =SL_2(\bar{\mathbb{F}}_q)$, ${\bf T}$  be the diagonal matrices and ${\bf U}$ be the strictly upper unitriangular matrices in $SL_2(\bar{\mathbb{F}}_q)$. Let ${\bf B}$ be the Borel subgroup generated by ${\bf T}$ and ${\bf U}$, which is the upper triangular matrices in $SL_2(\bar{\mathbb{F}}_q)$.
As before, let ${\bf N}$ be the normalizer of ${\bf T}$ in ${\bf G}$ and  $W= {\bf N}/{\bf T}$ be the  Weyl group.  We set
$\dot{s}=\begin{pmatrix}0 &1\\  -1 &0\end{pmatrix}$, which is the simple reflection of $W$.
There are two natural isomorphisms
$$h:   \bar{\mathbb{F}}^*_q\rightarrow{\bf T}, \   h(c)=\begin{pmatrix}c &0 \\0&c^{-1}\end{pmatrix}; \ \ \ \
\varepsilon : \bar{\mathbb{F}}_q\rightarrow{\bf U}, \ \  \varepsilon(x)= \begin{pmatrix}1&x\\0&1\end{pmatrix} $$
which satisfies $h(c)\varepsilon(x)h(c)^{-1}= \varepsilon(c^2x)$. The simple root
$\alpha: {\bf T}\rightarrow  \bar{\mathbb{F}}^*_q$ is given by $\alpha(h(c))= c^2$.
Moreover, one has that
\begin{align} \label{sus=xsty}
\dot{s} \varepsilon(x) \dot{s}=\displaystyle \varepsilon(-x^{-1}) \dot{s} h(-x) \varepsilon(-x^{-1}).
\end{align}
In the following we often denote $\theta(x): =\theta(h(x))$  simply. For any finite subset $X$ of ${\bf G}$, let $\overline{X}:=\displaystyle \sum_{x\in X}x \in\Bbbk{\bf G}$. This notation will be frequently used later.

In order to understand the irreducible  $\Bbbk {\bf G}$-modules with $\bf T$-stable lines, it is enough to
study the simple quotients of the induced module $\op{Ind}_{\bf T}^{\bf G} \Bbbk_\theta$ for each $\theta \in \widehat{\bf T}$.  It is not difficult to see that $\op{Ind}_{\bf T}^{\bf G} \Bbbk_\theta \cong  \op{Ind}_{\bf T}^{\bf G} \Bbbk_{\theta^s}$ as $\Bbbk {\bf G}$-modules.
We will consider the following two cases: (1) $\theta$ is trivial; (2) $\theta$ is nontrivial.  When $\theta$ is trivial,  it is easy to see  that
$$\op{Ind}_{\bf T} ^{\bf G} \Bbbk_{\text{tr}}\cong \op{Ind}_{\bf N} ^{\bf G} \Bbbk_{+} \oplus \op{Ind}_{\bf N} ^{\bf G} \Bbbk_{-},$$
where $\Bbbk_{+}$ is the trivial representation of ${\bf N}$ and $\Bbbk_{-}$ is the sign representation of $W$, which also can be regarded as a representation of  ${\bf N}$.

In our case ${\bf G} =SL_2(\bar{\mathbb{F}}_q)$, according to \cite[Theorem 3.1]{CD3}, $\mathbb{M}(\text{tr})=\op{Ind}_{\bf B} ^{\bf G} \Bbbk_{\text{tr}}$ has a unique submodule $\text{St}$ (the Steinberg module) and the corresponding quotient module is trivial. The Steinberg module $\text{St}= \Bbbk {\bf U} \eta_{s}$, where $\eta_s= (1-s) {\bf 1}_{\text{tr}}$. Then we have $\dot{s} \eta_s =-\eta_s$ and $\dot{s} \varepsilon(x)\eta_s =(\varepsilon(-x^{-1})-1) \eta_s$ for $x\ne 0$ by (\ref{sus=xsty}). For any nontrivial character $\theta\in \widehat{{\bf T}}$, $\mathbb{M}(\theta)=\op{Ind}_{\bf B}^{\bf G} \Bbbk_\theta $ is a simple $\Bbbk {\bf G}$-module. The main theorem of this section is as following.

\begin{theorem} \label{SimpleStaline}\normalfont
(1) The trivial $\Bbbk {\bf G}$-module $\Bbbk_{\text{tr}}$ is the unique simple quotient module of $\op{Ind}_{\bf N} ^{\bf G} \Bbbk_{+}$ and the Steinberg module $\text{St}$ is the  unique simple quotient module of $\op{Ind}_{\bf N} ^{\bf G} \Bbbk_{-}$.
(2) When $\theta$ is nontrivial,  $\op{Ind}_{\bf T}^{\bf G} \Bbbk_\theta$ has just two simple  quotient modules  $\mathbb{M}(\theta)$ and
$\mathbb{M}(\theta^s)$.
\end{theorem}

\begin{proof} [{Proof of Theorem \ref{SimpleStaline}(1)}]  For convenience,  we abbreviate $x{\bf 1}_{+}:=x\otimes{\bf 1}_{+}\in \op{Ind}_{\bf N} ^{\bf G} \Bbbk_{+}$ and
$x{\bf 1}_{-}:=x\otimes{\bf 1}_{-}\in \op{Ind}_{\bf N} ^{\bf G} \Bbbk_{-}$ for $x\in {\bf G}$, where ${\bf 1}_{+}$ (resp. ${\bf 1}_{-}$ ) is a fixed nonzero element in $\Bbbk_{+}$  (resp. $\Bbbk_{-}$).
Firstly we consider the simple quotient of $\op{Ind}_{\bf N} ^{\bf G} \Bbbk_{+}$. We construct a submodule of $\op{Ind}_{\bf N} ^{\bf G} \Bbbk_{+}$ as following
$$\mathbb{M}_{+}=\{\sum_{g\in{\bf G}}a_g g {\bf 1}_{+} \mid\sum_{g\in{\bf G}}a_g=0\}.$$
Then it is easy to see that $\op{Ind}_{\bf N} ^{\bf G} \Bbbk_{+} /\mathbb{M}_{+}$ is the trivial $\Bbbk {\bf G}$-module.
Now let $\xi \notin \mathbb{M}_{+}$ which has the following expression
$$\xi=  \sum_{x, y\in \bar{\mathbb{F}}_q } a_{x,y} \varepsilon(x)\dot{s}\varepsilon(y){\bf 1}_{+},$$
where $\displaystyle \sum_{x, y\in \bar{\mathbb{F}}_q } a_{x,y} \ne 0.$ Firstly there exists an integer $m\in \mathbb{N}$ such that $x, y\in \mathbb{F}_{q^m}$ when  $a_{x,y}\ne 0$. Now let $n>m$ with $m | n$ and  $u\in \mathbb{F}_{q^n} \setminus \mathbb{F}_{q^m}$. We consider the element
$\eta= \dot{s} \varepsilon(u) \xi $, which has the form
$$\eta = \displaystyle \sum_{x, y\in \bar{\mathbb{F}}_q } a_{x,y} \varepsilon(-(u+x)^{-1})\dot{s}\varepsilon((u+x)- (u+x)^2y){\bf 1}_{+}. $$
Choose $\mathfrak{D}_{q^n}\subset {\bf T}$ such that $\alpha:\mathfrak{D}_{q^n}\rightarrow  \mathbb{F}^*_{q^n}$ is a bijection. Thus  it is easy to check the element
$$\overline{\mathfrak{D}_{q^n}}\ \overline{U_{q^n}}\eta =(\sum_{x, y\in \bar{\mathbb{F}}_q } a_{x,y})\overline{U_{q^n}} \dot{s} \overline{U^*_{q^n}}{\bf 1}_{+} \in \Bbbk {\bf G}{\bf 1}_{+}$$
which implies that $\overline{U_{q^n}} \dot{s} \overline{U^*_{q^n}}{\bf 1}_{+}\in \Bbbk {\bf G}\xi$.

On the other hand, we consider the element
$$\overline{G_{q^n}}\xi =(\sum_{x, y\in \bar{\mathbb{F}}_q } a_{x,y})\overline{G_{q^n}}{\bf 1}_{+} \in \Bbbk {\bf G}\xi.$$
Noting that
$$\overline{G_{q^n}}{\bf 1}_{+}=(q^n-1) (2 \overline{U_{q^n}}{\bf 1}_{+} +  \overline{U_{q^n}} \dot{s} \overline{U^*_{q^n}}{\bf 1}_{+}) \in \Bbbk {\bf G}\xi,$$
therefore we have $\overline{U_{q^n}}{\bf 1}_{+} \in \Bbbk {\bf G}\xi$ and hence ${\bf 1}_{+} \in \Bbbk {\bf G}\xi $ by {\cite[Lemma 2.6]{D2}}.
Thus for any element $\xi \notin \mathbb{M}_{+}$,  we have $\Bbbk {\bf G}\xi =\op{Ind}_{\bf N} ^{\bf G} \Bbbk_{+}$. So, the trivial module is the unique simple quotient module of $\op{Ind}_{\bf N} ^{\bf G} \Bbbk_{+}$.

\smallskip

Now we consider the simple  quotient modules of $\op{Ind}_{\bf N} ^{\bf G} \Bbbk_{-}$. For convenience, we denote by
$$\Lambda(z)= (\dot{s} \varepsilon(z) +1- \varepsilon(-{z}^{-1})) {\bf 1}_{-}$$
for each $z\in \bar{\mathbb{F}}^*_q$. Then we have $h(c)\Lambda(z)= \Lambda(c^{-2}z)$ for any $h(c) \in {\bf T}$. Moreover, it is easy to check that
\begin{align} \label{M-sub}
 \varepsilon(z)\Lambda(z^{-1})=-\Lambda(-z^{-1}) ,\quad
 \dot{s}\Lambda(z)=-\Lambda(-z^{-1}),  \quad \quad   \quad\\ \notag \dot{s}\varepsilon(x)\Lambda(y)=\varepsilon(-x^{-1})\Lambda(x(xy-1))+\Lambda(x)-\Lambda(y^{-1}(xy-1)),
\end{align}
where $xy\ne 1$. Let $\mathbb{M}_{-}$ be the submodule of $\op{Ind}_{\bf N} ^{\bf G} \Bbbk_{-}$ which is generated by $\Lambda(z), z\in \bar{\mathbb{F}}^*_q$. Since we have the equation
$$\dot{s}\varepsilon(z) \eta_s= (\varepsilon(-{z}^{-1})-1)\eta_s$$ in
the Steinberg module $\text{St}=\Bbbk {\bf U} \eta_s$. Thus it is not difficult to see that $\op{Ind}_{\bf N} ^{\bf G} \Bbbk_{-}/ \mathbb{M}_{-} $ is isomorphic to the Steinberg module.

Let $\zeta\in \op{Ind}_{\bf N} ^{\bf G} \Bbbk_{-}$ which is not in $ \mathbb{M}_{-} $. According to (\ref{M-sub}), then $\zeta$ has the following expression
$$\zeta= \sum_{xy\ne 1}a_{x,y} \varepsilon(x)\Lambda(y) {\bf 1}_{-} + \sum_{z\in \bar{\mathbb{F}}^*_q} b_{z} \Lambda(z){\bf 1}_{-}+ \sum_{u\in \bar{\mathbb{F}}_q} c_u\varepsilon(u) {\bf 1}_{-}  ,$$
where $c_u \ne 0$ for some $u$. There exists an integer $m\in \mathbb{N}$ such that $x, y,z,u\in \mathbb{F}_{q^m}$ when  $a_{x,y}\ne 0, b_{z}\ne 0$ and $c_u\ne 0$. Without lost of generality, we can assume that $c_{0}\ne 0$. Moreover, we can assume that $\displaystyle \sum_{u\in \bar{\mathbb{F}}_q} c_u \ne 0$. Otherwise, we can consider $\dot{s} \zeta$ instead of $\zeta$. Indeed, if we write
$$\dot{s}\zeta=\sum_{x y\ne 1}a'_{x,y} \varepsilon(x)\Lambda(y) {\bf 1}_{-} + \sum_{z\in \bar{\mathbb{F}}^*_q} b'_{z} \Lambda(z){\bf 1}_{-}+ \sum_{u\in \bar{\mathbb{F}}_q} c'_{u} \varepsilon(u){\bf 1}_{-},$$
it is easy to see that $\displaystyle \sum_{u\in \bar{\mathbb{F}}_q} c'_{u}=-c_0$ which is nonzero.
For the convenience of later discussion, we denote by
$$A= \sum_{xy\ne 1}a_{x,y},\quad B= \sum_{z\in \bar{\mathbb{F}}^*_q} b_{z}, \quad C= \sum_{u\in \bar{\mathbb{F}}_q} c_u. $$
Note that $C$ is nonzero by our assumption.

Now let $n>m$ with $m | n$  and $v \in \mathbb{F}_{q^n} \setminus \mathbb{F}_{q^m}$. We consider the element
$$\dot{s} \varepsilon(v)\zeta:= \sum_{x y\ne 1}f_{x,y} \varepsilon(x)\Lambda(y) {\bf 1}_{-} + \sum_{z\in \bar{\mathbb{F}}^*_q} g_{z} \Lambda(z){\bf 1}_{-}+ \sum_{u\in \bar{\mathbb{F}}_q} h_{u} \varepsilon(u){\bf 1}_{-}.$$
Then by (\ref{M-sub}), we get
\begin{align}\label{Eq3.1}
\sum_{x y\ne 1}f_{x,y}=A+B, \quad \sum_{z\in \bar{\mathbb{F}}^*_q} g_{z}=C, \quad \sum_{u\in \bar{\mathbb{F}}_q} h_{u}= 0.
\end{align}
Choose $\mathfrak{D}_{q^n}\subset {\bf T}$ such that $\alpha:\mathfrak{D}_{q^n}\rightarrow  \mathbb{F}^*_{q^n}$ is a bijection. Combining (\ref{Eq3.1}), it is not difficult to get
\begin{align}\label{Eq3.2}
\overline{\mathfrak{D}_{q^n}}\ \overline{U_{q^n}}\dot{s}\varepsilon(v)\zeta=(A+B+C) \overline{U_{q^n}} \dot{s} \overline{U^*_{q^n}}{\bf 1}_{-} \in \Bbbk {\bf G}\zeta.
\end{align}
On the other hand, we also have
\begin{align}\label{Eq3.3}
\overline{\mathfrak{D}_{q^n}}\ \overline{U_{q^n}}\zeta= (A+B)\overline{U_{q^n}} \dot{s} \overline{U^*_{q^n}}{\bf 1}_{-} + (q^n-1)C \overline{U_{q^n}}{\bf 1}_{-} \in \Bbbk {\bf G}\zeta.
\end{align}
If $A+B+C\ne 0$ and noting that $C\ne 0$, then we get $\overline{U_{q^n}}{\bf 1}_{-} \in \Bbbk {\bf G}\zeta$ using (\ref{Eq3.2}) and (\ref{Eq3.3}).
Now assume that $A+B+C=0$, then by (\ref{Eq3.3}), we have
$$\zeta':=\overline{U_{q^n}} \dot{s} \overline{U^*_{q^n}}{\bf 1}_{-} - (q^n-1) \overline{U_{q^n}}{\bf 1}_{-} \in \Bbbk {\bf G}\zeta.$$
Noting that $\zeta'$ can be written as
\begin{align*}\zeta' &\  =\sum_{z\in \mathbb{F}^*_{q^n}}\overline{U_{q^n}} \Lambda(z){\bf 1}_{-} - (q^n-1) \overline{U_{q^n}}{\bf 1}_{-}\\
&\  =\sum_{x,y\in \mathbb{F}^*_{q^n}, xy\ne 1} \varepsilon(x)\Lambda(y) {\bf 1}_{-} -(q^n-1) \overline{U_{q^n}}{\bf 1}_{-}
\end{align*}
using (\ref{M-sub}), thus it is not difficult to see that the sum of the coefficients in $\zeta'$ is $-2(q^n-1)$. So we can discuss $\zeta'$ instead of $\zeta$ form the beginning and also have $\overline{U_{q^n}}{\bf 1}_{-} \in \Bbbk {\bf G}\zeta$.

In conclusion, we have $\overline{U_{q^n}}{\bf 1}_{-} \in \Bbbk {\bf G}\zeta$ and thus  we get ${\bf 1}_{-} \in \Bbbk {\bf G}\zeta $ by {\cite[Lemma 2.6]{D2}}. Then for any element $\zeta \notin \mathbb{M}_{-}$,  we see that $\Bbbk {\bf G}\zeta =\op{Ind}_{\bf N} ^{\bf G} \Bbbk_{-}$. So, the Steinberg module is the unique simple quotient module of $\op{Ind}_{\bf N} ^{\bf G} \Bbbk_{-}$.
Hence Theorem \ref{SimpleStaline}(1) is proved.
\end{proof}

Before the proof of  Theorem \ref{SimpleStaline}(2), we need to introduce some properties of the nontrivial group homomorphisms from $\bar{\mathbb{F}}^*_q$  to $\Bbbk^*$, where $\Bbbk$ is an algebraically closed field of characteristic $0$ as before.

\begin{proposition} \normalfont \label{detinf}
Let $\lambda : \bar{\mathbb{F}}^*_q \rightarrow \Bbbk^*$ be a nontrivial  group homomorphism and $u_1, u_2, \dots, u_n\in \bar{\mathbb{F}}^*_q$, which are different from each other. Let $x_1, x_2,\dots, x_n$ be $n$ variables with values in $\bar{\mathbb{F}}^*_q$. Denote by $\mathbf{u}=(u_1,  u_2, \dots, u_n)$ and by $\mathbf{x}=(x_1,  x_2, \dots, x_n)$.  Let $A^\lambda_{\mathbf{u}}(\mathbf{x})$ be the matrix whose entry in  row $i$ and  column $j$ is $\lambda(x_i+u_j)$. Thus there exist infinitely many $\mathbf{x}\in (\bar{\mathbb{F}}^*_q)^n $
such that $\det A^\lambda_{\mathbf{u}}(\mathbf{x}) \ne 0$.
\end{proposition}

\begin{proof} The lemma is obvious when $n=1$. Assume that this proposition holds when $n\leq m$. Now we consider the case for $n=m+1$. We set
$$\mathbf{u}^k=(u_1, \dots,u_{k-1},u_{k+1}, \dots,  u_{m+1}),\ \ \ \ \mathbf{x}'= (x_1,  x_2, \dots, x_m).$$
By cofactor expansion respect to the last row, we have
$$\det A^\lambda_{\mathbf{u}}(\mathbf{x}) = \sum_{k=1}^{m+1} (-1)^{m+1+k}\lambda(x_{m+1}+u_k) \det A^\lambda_{\mathbf{u}^k}(\mathbf{x}').$$
For each fixed integer $k=1,2, \dots, m+1$, we have infinitely  many $\mathbf{x}'\in (\bar{\mathbb{F}}^*_q)^m $ such that $\det A^\lambda_{\mathbf{u}^k}(\mathbf{x}') \ne 0$ by the inductive hypothesis. Hence it is enough to show that for any element $a_1, a_2, \dots, a_{m+1} \in \Bbbk^*$, there exists  an element $x_0\in \bar{\mathbb{F}}^*_q$ such that
\begin{align} \label{nonzeroEQ}
a_1 \lambda(x_0+u_1) + a_2 \lambda(x_0+u_2) + \dots + a_{m+1} \lambda(x_0+u_{m+1}) \ne 0.
\end{align}

For an integer $s$, we let $\Gamma_s= \lambda (\mathbb{F}^*_{q^s})$,  which is a finite cyclic group in $\Bbbk^*$. Set $b_k= - a_k/a_{m+1}$ and we deal with  the following equation
$$b_1 y_1 +b_2 y_2 +\dots +b_m y_m=1$$
with solutions $\mathbf{y}=(y_1, y_2,\dots, y_m) \in (\Gamma_s)^m$. Denote the solution set of this equation  by $S(\mathbf{b},\Gamma_s)$. Then by \cite[Theorem 1.1]{ESS}, for any $s\in \mathbb{N}$, we have $|S(\mathbf{b}, \Gamma_s)|\leq \mathcal{C}(m)$ for some integer $\mathcal{C}(m)$ only depends on $m$.
Since $\lambda : \bar{\mathbb{F}}^*_q \rightarrow \Bbbk^*$ is nontrivial and then we can choose the integer $s$ large enough such that
$$|\Gamma_s /\ker \lambda|> \mathcal{C}(m)+m.$$
Thus there exists  $x_0\in \bar{\mathbb{F}}^*_q$ such that $x_0+u_j \ne 0$ for any $j=1,2, \dots,m+1$ and (\ref{nonzeroEQ}) holds. The proposition is proved.
\end{proof}

According to the above proof of Proposition \ref{detinf}, we get the following corollary immediately.

\begin{corollary} \label{nonzerocoeff} \normalfont
Let  $\lambda : \bar{\mathbb{F}}^*_q \rightarrow \Bbbk^*$ be a nontrivial  group homomorphism and $u_1, u_2, \dots, u_n\in \bar{\mathbb{F}}^*_q$,  which are different from each other.
Given $a_1, a_2,\dots, a_n \in \Bbbk^*$, then there exists infinitely many elements $x\in  \bar{\mathbb{F}}^*_q$ such that
$$a_1\lambda (x+u_1)+ a_2 \lambda(x+u_2)+\dots + a_n\lambda(x+u_n)\ne 0.$$

\end{corollary}

\begin{remark} \normalfont
Proposition \ref{detinf} and Corollary \ref{nonzerocoeff} do not hold when $\Bbbk=\bar{\mathbb{F}}_q$.
Indeed, we can choose $q$ elements $u_1, u_2, \dots, u_q\in \bar{\mathbb{F}}^*_q$,  which are different from each other and satisfy
$u^q_1+u^q_2+\dots +u^q_q=0.$
Let $\lambda : \bar{\mathbb{F}}^*_q \rightarrow \bar{\mathbb{F}}^*_q$ be a group homomorphism such that $\lambda(x)=x^q$. Then it is easy to see that
$$\lambda (x+u_1)+  \lambda(x+u_2)+\dots + \lambda(x+u_q)=0$$
for any $x\in \bar{\mathbb{F}}^*_q$ with $x+u_i \ne 0$, which is a counter example to  Proposition \ref{detinf} and Corollary \ref{nonzerocoeff}. However I guess that Proposition \ref{detinf} and Corollary \ref{nonzerocoeff} still hold when $\op{char}\Bbbk\neq \op{char} \bar{\mathbb{F}}_q$. Other methods need to be developed.

\end{remark}

With the above preparations, now we prove Theorem \ref{SimpleStaline}(2).

\begin{proof} [{Proof of Theorem \ref{SimpleStaline}(2)}]

Let $\varphi_e: \op{Ind}_{\bf T}^{\bf G} \Bbbk_\theta \rightarrow \mathbb{M}(\theta)$ be the natural morphism such that $\varphi_e ({\bf 1}_{\theta})= {\widehat{\bf 1}}_{\theta}$, where ${\widehat{\bf 1}}_{\theta}$ is a fixed nonzero element in $\mathbb{M}(\theta)_{\theta}$. Let $\varphi_s: \op{Ind}_{\bf T}^{\bf G} \Bbbk_\theta \rightarrow \mathbb{M}(\theta^s) $ be the morphism such that $\varphi_s ({\bf 1}_{\theta})=\dot{s}{\widehat{\bf 1}}_{\theta^s}$. Both $\mathbb{M}(\theta)$ and $\mathbb{M}(\theta^s)$ are the simple quotient modules of $\op{Ind}_{\bf T}^{\bf G}\Bbbk_{\theta}$. Now let $\xi$ be a nonzero element in $\op{Ind}_{\bf T}^{\bf G}  \Bbbk_\theta $. We will show that ${\Bbbk {\bf G}} \xi= \op{Ind}_{\bf T}^{\bf G}  \Bbbk_\theta  $ if $\xi \notin  \ker \varphi_e$ and $\xi \notin  \ker \varphi_s$. Thus $\op{Ind}_{\bf T}^{\bf G} \Bbbk_\theta$  has no other simple quotient modules except $\mathbb{M}(\theta)$ and $\mathbb{M}(\theta^s)$.

\smallskip

\noindent Claim ($\clubsuit$): One has that $\overline{U_{ q^n}}\dot{s} \overline{U^*_{ q^n}}{\bf 1}_{\theta}\in \Bbbk {\bf G}\xi $ when $\xi \notin  \ker \varphi_e$.

\smallskip

\noindent Proof of Claim ($\clubsuit$): Now we write  $\xi$ as following
$$\xi= \sum_{x\in \bar{\mathbb{F}}_q}a_x \varepsilon(x){\bf 1}_{\theta}+ \sum_{y,z\in \bar{\mathbb{F}}_q} b_{y,z} \varepsilon(y)\dot{s}\varepsilon(z){\bf 1}_{\theta}.$$
Noting that  $\xi \notin  \ker \varphi_e$, the following equation
$$\varphi_e(\xi)= \sum_{x\in \bar{\mathbb{F}}_q}a_x {\widehat{\bf 1}}_{\theta}+ \sum_{y,z\in \bar{\mathbb{F}}_q} b_{y,z} \varepsilon(y)\dot{s}{\widehat{\bf 1}}_{\theta}\ne 0$$
tells us that $\displaystyle \sum_{x\in \bar{\mathbb{F}}_q}a_x \ne 0$ or $\displaystyle  \sum_{z\in \bar{\mathbb{F}}_q} b_{y,z}\ne 0$ for some $y$. Without lost of generality, we can assume that $\displaystyle \sum_{x\in \bar{\mathbb{F}}_q}a_x \ne 0$. Otherwise, it is enough to consider the element $\dot{s}^{-1}\varepsilon(-y_0) \xi$ instead of $\xi$. Moreover, we can also  assume that
$\displaystyle \sum_{x\in \bar{\mathbb{F}}_q}a_x+ \sum_{y,z\in \bar{\mathbb{F}}_q} b_{y,z}\ne 0.$
Otherwise, we can consider the element $t\xi$  for some $t\in {\bf T}$ instead of $\xi$. Indeed, it is easy to see that the sum of the coefficients in  $t\xi$  is
\begin{align} \label{coeff}
\displaystyle \theta(t)\sum_{x\in \bar{\mathbb{F}}_q}a_x+ \theta^s(t)\sum_{y,z\in \bar{\mathbb{F}}_q} b_{y,z}.
\end{align}
Noting that  $\displaystyle \sum_{x\in \bar{\mathbb{F}}_q}a_x$ and $\displaystyle  \sum_{y,z\in \bar{\mathbb{F}}_q} b_{y,z}$ are nonzero, we can choose one $t\in {\bf T}$ such that (\ref{coeff}) is nonzero since $\theta$ is nontrivial.

With the assumption that
$$\sum_{x\in \bar{\mathbb{F}}_q}a_x \ne 0  \quad \text{and} \quad \sum_{x\in \bar{\mathbb{F}}_q}a_x+ \sum_{y,z\in \bar{\mathbb{F}}_q} b_{y,z}\ne 0,$$
we show that $\displaystyle \overline{U_{q^n}} \dot{s} \overline{U^*_{q^n}} \in \Bbbk {\bf G}\xi$ for some $n\in \mathbb{N}$. Firstly there exists $m\in \mathbb{N}$ such that $x, y, z\in \mathbb{F}_{q^m}$ when $a_x \ne 0$ and $b_{y,z}\ne 0$. Let $n>m$ and $m| n$. Given an element $u\in \mathbb{F}_{q^n}\setminus \mathbb{F}_{q^m}$, then the element $\eta:=\dot{s} \varepsilon(u)\xi $ has the following form
\begin{align*}&\
 \eta=\sum_{x\in \bar{\mathbb{F}}_q}a_x \dot{s}\varepsilon(u+x){\bf 1}_{\theta}+ \sum_{y,z\in \bar{\mathbb{F}}_q} b'_{y,z} \varepsilon(-(u+y)^{-1})\dot{s} \varepsilon((u+y)^2z-(u+y)){{\bf 1}}_{\theta},
\end{align*}
where $b'_{y,z}=  b_{y,z}\theta(h(-(u+y)))$. Thus if we denote by $\eta$ as following
$$\eta=\sum_{x,y\in \bar{\mathbb{F}}_q} f_{x,y} \varepsilon(x)\dot{s}\varepsilon(y){\bf 1}_{\theta},$$
then we get  $f_x:=\displaystyle \sum_{y\in \bar{\mathbb{F}}_q} f_{x,y}\ne 0$ for some $x$.
Let $v\in  \bar{\mathbb{F}}^*_q$ such that $v+x \ne 0$ when $f_{x,y}\ne 0$ and
we consider the element
$$\zeta:=\dot{s} \varepsilon(v)\eta= \sum_{x,y\in \bar{\mathbb{F}}_q} f_{x,y}\theta(h(-(v+x))) \varepsilon(-(v+x)^{-1})\dot{s} \varepsilon((v+x)^2y-(v+x)){{\bf 1}}_{\theta}.$$
By Lemma \ref{nonzerocoeff}, we can choose some element $v$ such that the elements
$(v+x)^2y-(v+x)$ are nonzero when $f_{x,y}\ne 0$ and this element  $v$ also makes
$$f:=\sum_{x\in \bar{\mathbb{F}}_q} f_{x}\theta(h(-(v+x))) \ne 0.$$
As before we choose $\mathfrak{D}_{q^n}\subset {\bf T}$ such that $\alpha:\mathfrak{D}_{q^n}\rightarrow  \mathbb{F}^*_{q^n}$ is a bijection. Therefore  we get
$$\sum_{t\in \mathfrak{D}_{q^n}} \theta^s(t)^{-1} t  \overline{U_{ q^n}}\zeta= f \overline{U_{ q^n}}\dot{s} \overline{U^*_{ q^n}}{\bf 1}_{\theta}$$
which implies that $\overline{U_{ q^n}}\dot{s} \overline{U^*_{ q^n}}{\bf 1}_{\theta}\in \Bbbk {\bf G}\xi $.
Claim ($\clubsuit$) is proved.

\smallskip

\noindent Claim ($\spadesuit$): When  $\xi \notin  \ker \varphi_s$, one has an element
$$\xi'= \sum_{x\in \bar{\mathbb{F}}_q}a_x \varepsilon(x){\bf 1}_{\theta}+ \sum_{y\in \bar{\mathbb{F}}_q} b_{y} \varepsilon(y)\dot{s}{\bf 1}_{\theta}+ \sum_{u\in \bar{\mathbb{F}}_q, v\in \bar{\mathbb{F}}^*_q } c_{u,v} \varepsilon(u)\dot{s}\varepsilon(v){\bf 1}_{\theta} \in \Bbbk {\bf G}\xi$$
such that $\displaystyle \sum_{y\in \bar{\mathbb{F}}_q} b_{y} \ne 0$.

\smallskip

\noindent Proof of Claim ($\spadesuit$):  Firstly we study the form of the elements in $\ker \varphi_s$. Let
$$\varpi_1= \sum_{x\in \bar{\mathbb{F}}_q} f_x \varepsilon(x)(e+ \sum_{u\in \bar{\mathbb{F}}_q} f_{x,u} \theta^s(-u)\varepsilon(u)\dot{s} \varepsilon(u^{-1}) ){\bf 1}_{\theta},$$
where $\displaystyle \sum_{u\in \bar{\mathbb{F}}_q} f_{x,u} +1 =0$ for any $x\in  \bar{\mathbb{F}}_q$ and set
$$\varpi_2= \sum_{y\in \bar{\mathbb{F}}_q} \varepsilon(y) \sum_{v\in \bar{\mathbb{F}}_q} g_{y,v} \theta^s(-v)\varepsilon(v)\dot{s} \varepsilon(v^{-1}) {\bf 1}_{\theta},$$
where $\displaystyle \sum_{v\in \bar{\mathbb{F}}_q} g_{y,v} =0$ for any $y\in  \bar{\mathbb{F}}_q$. Moreover we let
$$ \varpi_3 =\sum_{z\in \bar{\mathbb{F}}_q} h_{z} \varepsilon(z) \dot{s} {\bf 1}_{\theta},\quad \text{where} \ \sum_{z\in \bar{\mathbb{F}}_q} h_{z}=0.$$
It is easy to check that $\varpi_1, \varpi_2, \varpi_3 \in \ker \varphi_s$. Denote by $\Omega_i$ the set of the elements with the form of $\varpi_i$ for $i=1,2,3$.
Then it is not difficult to see that each element of $\ker \varphi_s$  is a linear combination of the elements in $\Omega_1, \Omega_2$ and $\Omega_3$.

Now let $\xi$ be the following
$$\xi= \sum_{x\in \bar{\mathbb{F}}_q}\alpha_x \varepsilon(x){\bf 1}_{\theta}+ \sum_{y\in \bar{\mathbb{F}}_q} \beta_{y} \varepsilon(y)\dot{s}{\bf 1}_{\theta}+ \sum_{u\in \bar{\mathbb{F}}_q, v\in \bar{\mathbb{F}}^*_q } \gamma_{u,v} \varepsilon(u)\dot{s}\varepsilon(v){\bf 1}_{\theta}$$
If $\displaystyle \sum_{y\in \bar{\mathbb{F}}_q} \beta_{y} \ne 0$, then $\xi$ has already satisfied our requirement. Otherwise, we can write
$$\xi= \varpi + \sum_{u\in \bar{\mathbb{F}}_q, v\in \bar{\mathbb{F}}^*_q } \tau_{u,v} \varepsilon(u)\dot{s}\varepsilon(v){\bf 1}_{\theta},\quad \tau_{u,v}\in \Bbbk$$
such that  $ \varpi \in \ker \varphi_s$. Since $\xi \notin  \ker \varphi_s$, the element
$$\eta:= \displaystyle \sum_{u\in \bar{\mathbb{F}}_q, v\in \bar{\mathbb{F}}^*_q } \tau_{u,v} \varepsilon(u)\dot{s}\varepsilon(v){\bf 1}_{\theta}$$ is not in $ \ker \varphi_s$ and in particular, it is not in $\Omega_2$. Now for $x\in \bar{\mathbb{F}}_q$, set
$$\Xi_x= \{(u,v) \mid \tau_{u,v} \ne 0 \ \text{and}  \ (u-x)v=1   \}.$$
Since $\eta \notin  \ker \varphi_s$, there exist an element $x_0$ such that
\begin{align} \label{Eq3.4}
\sum_{(u,v)\in \Xi_{x_0}}\tau_{u,v} \theta^{s}(-v) \ne 0
\end{align}
by some easy computation.
Now we consider  the following element
$$\xi'= \dot{s}  \varepsilon(-x_0)\xi=\dot{s}  \varepsilon(-x_0)(\varpi + \eta ).$$
Firstly, we have $\dot{s}  \varepsilon(-x_0) \varpi \in  \ker \varphi_s$, which is a linear combination of elements in $\Omega_1, \Omega_2$ and $\Omega_3$. For the second part $\dot{s}  \varepsilon(-x_0) \eta$, it is easy to check that when $(u,v)\notin \Xi_{x_0}$, the element
$$\dot{s}  \varepsilon(-x_0)\varepsilon(u)\dot{s}\varepsilon(v){\bf 1}_{\theta}=  \varepsilon(u')\dot{s}\varepsilon(v'){\bf 1}_{\theta}$$
for some $u'\in \bar{\mathbb{F}}_q$ and $v' \in \bar{\mathbb{F}}^*_q$. On the other hand, for $(u,v)\in \Xi_{x_0}$, we have
$$\dot{s}  \varepsilon(-x_0)\sum_{(u,v)\in \Xi_{x_0}}\tau_{u,v}\varepsilon(u)\dot{s}\varepsilon(v){\bf 1}_{\theta} = \sum_{(u,v)\in \Xi_{x_0}}\tau_{u,v}\theta(-v^{-1})\varepsilon(-v)\dot{s}{\bf 1}_{\theta}.$$
Noting that $\theta(-v^{-1})= \theta^{s}(-v)$ and using (\ref{Eq3.4}), it is easy to check that $\xi'= \dot{s}  \varepsilon(-x_0)\xi \in \Bbbk {\bf G}\xi$ satisfies our requirement.

\smallskip

With Claim ($\clubsuit$) and Claim ($\spadesuit$), now we give the proof of Theorem \ref{SimpleStaline}(2). Let $\xi$ be an element such that  $\xi \notin  \ker \varphi_e$ and  $\xi \notin  \ker \varphi_s$. By Claim ($\spadesuit$), there exists an element
$$\xi'= \sum_{x\in \bar{\mathbb{F}}_q}a_x \varepsilon(x){\bf 1}_{\theta}+ \sum_{y\in \bar{\mathbb{F}}_q} b_{y} \varepsilon(y)\dot{s}{\bf 1}_{\theta}+ \sum_{u\in \bar{\mathbb{F}}_q, v\in \bar{\mathbb{F}}^*_q } c_{u,v} \varepsilon(u)\dot{s}\varepsilon(v){\bf 1}_{\theta} \in \Bbbk {\bf G}\xi$$
such that $\displaystyle \sum_{y\in \bar{\mathbb{F}}_q} b_{y} \ne 0$. We  denote by
$$A= \sum_{x\in \bar{\mathbb{F}}_q}a_x,\quad  B:=\displaystyle \sum_{y\in \bar{\mathbb{F}}_q} b_{y},\quad  C= \sum_{u\in \bar{\mathbb{F}}_q, v\in \bar{\mathbb{F}}^*_q } c_{u,v}$$ for simple.
Let $n\in \mathbb{N}$ such that $x,y,u,v\in \mathbb{F}_{q^n}$ when $a_x\ne 0, b_{y}\ne 0$ and $c_{u,v}\ne 0$. As before we choose $\mathfrak{D}_{q^n}\subset {\bf T}$ such that $\alpha:\mathfrak{D}_{q^n}\rightarrow  \mathbb{F}^*_{q^n}$ is a bijection. Hence it is not difficult to see that
$$\sum_{t\in \mathfrak{D}_{q^n}}\theta^s(t)^{-1} t \overline{U_{ q^n}}\xi'=A \sum_{t\in \mathfrak{D}_{q^n}}\theta(t^2){\bf 1}_{\theta} + (q^n-1)B \overline{U_{ q^n}}\dot{s} {\bf 1}_{\theta} + C \overline{U_{ q^n}}\dot{s} \overline{U^*_{ q^n}}{\bf 1}_{\theta},$$
which is in $\Bbbk {\bf G}\xi$. When $n$ is big enough and $\theta$ is nontrivial on $\mathbb{F}^*_{q^n}$, we have $\displaystyle \sum_{t\in \mathfrak{D}_{q^n}}\theta(t^2)=0$.
By Claim ($\clubsuit$), we have $\overline{U_{ q^n}}\dot{s} \overline{U^*_{ q^n}}{\bf 1}_{\theta}\in \Bbbk {\bf G}\xi $  and then we get $\overline{U_{ q^n}}\dot{s} {\bf 1}_{\theta} \in  \Bbbk {\bf G}\xi$ since $B$ is nonzero. So using {\cite[Lemma 2.6]{D2}},  we have ${\bf 1}_{\theta} \in  \Bbbk {\bf G}\xi$ and thus $\Bbbk {\bf G}\xi=\op{Ind}_{\bf T}^{\bf G} \Bbbk_\theta$.
Therefore the induced module $\op{Ind}_{\bf T}^{\bf G} \Bbbk_\theta$   has only two  simple  quotient modules  $\mathbb{M}(\theta)$ and $\mathbb{M}(\theta^s)$. The theorem is proved.

\end{proof}

\section{Abelian category $\mathscr{C}(\bf G)$}

In this section, we show that the category $\mathscr{C}(\bf G)$ is an abelian category for  ${\bf G} =SL_2(\bar{\mathbb{F}}_q)$. For convenience, we denote
$$\text{Irr}({\bf G})= \{ \Bbbk_{\text{tr}}, \text{St}, \mathbb{M}(\theta)\mid \theta \in \widehat{{\bf T}} \ \text{is nontrivial}\}.$$
This is the set of the simple objects in $\mathscr{C}(\bf G)$. By \cite[Theorem 2.5]{CD4}, the induced module $\text{Ind}_{\bf T}^{\bf B} \Bbbk_{\theta}$ has a unique simple $\Bbbk {\bf B}$-submodule $$\mathcal{S}_{\theta}=\{\displaystyle \sum_{x\in {\bar{\mathbb{F}}_q }} a_x \varepsilon(x) {\bf 1}_{\theta}\mid \sum_{x\in {\bar{\mathbb{F}}_q }}  a_x=0\}$$ and the corresponding quotient module is $ \Bbbk_{\theta}$.

\begin{lemma} \label{Notmodule}\normalfont Let $\mathcal{S}=\mathcal{S}_{\text{tr}}$ be the unique simple $\Bbbk {\bf B}$-submodule of $\text{Ind}_{\bf T}^{\bf B} \text{tr}$. Then $\mathcal{S}^n$ (the direct sum of $n$-copies of $\mathcal{S}$) can not assign a $\Bbbk {\bf G}$-module structure for any $n\in \mathbb{N}$.
\end{lemma}

\begin{proof}
Firstly, each element $\xi \in \mathcal{S}^n$ has the following form
$$\xi=\sum_{j,\mu} f_{j,\mu}(e-\varepsilon(x_{j,\mu}))\lambda_j, \quad \text{where} \ x_{j,\mu}\in \bar{\mathbb{F}}^*_q, \ f_{j,\mu}\in \Bbbk,$$
where each $\lambda_j$ denotes the trivial character for $j=1,2,\dots, n$. Suppose $\mathcal{S}^n $ has a $\Bbbk {\bf G}$-module structure and for each $i=1,2, \dots, n$, we set
$$\dot{s}(e-\varepsilon(1))\lambda_i= \sum_{j,\mu} g^i_{j,\mu}(e-\varepsilon(x^i_{j,\mu}))\lambda_j, \quad \text{where} \ x^i_{j,\mu}\in \bar{\mathbb{F}}^*_q, \ g^i_{j,\mu}\in \Bbbk.$$
Using $t\in {\bf T}$ to act on both sides, then it is not difficult to see that for $y\in \bar{\mathbb{F}}^*_q$, we have
\begin{align*}
\dot{s}(e-\varepsilon(y))\lambda_i= \sum_{j,\mu} g^i_{j,\mu}(e-\varepsilon(y^{-1}x^i_{j,\mu}))\lambda_j.
\end{align*}
Therefore we have
\begin{align} \label{Eq4.1}
&\ \dot{s}\varepsilon(z)\dot{s}(e-\varepsilon(1))\lambda_i = \sum_{j,\mu} g^i_{j,\mu} \dot{s} (\varepsilon(z)-\varepsilon(x^i_{j,\mu}+z))\lambda_j  \\ \notag
&\ = \sum_{j,\mu} g^i_{j,\mu} \sum_{k,\nu}  g^j_{k,\nu} (\varepsilon(z^{-1} x^j_{k,\nu})- \varepsilon((x^i_{j,\mu}+z)^{-1} x^j_{k,\nu})) \lambda_k
\end{align}
for any element $z\in \bar{\mathbb{F}}^*_q$ such that  $x^i_{j,\mu}+z \ne 0$.
On the other hand, since $\dot{s}\varepsilon(z)\dot{s}=  \varepsilon(-z^{-1}) \dot{s} h(-z) \varepsilon(-z^{-1})$, we get
\begin{align} \label{Eq4.2}
&\ \dot{s}\epsilon(z)\dot{s}(e-\varepsilon(1))\lambda_i =  \varepsilon(-z^{-1}) \dot{s} (\varepsilon(-z)- \varepsilon(z^2-z)) \lambda_i \\ \notag
&\  = \sum_{k,\nu} g^i_{k,\nu} (\varepsilon(-z^{-1}(1+x^i_{k,\nu}))-\varepsilon(z^{-1}(z-1)^{-1}(x^i_{k,\nu}-z+1)) )\lambda_k.
\end{align}
Combining the above two equations (\ref{Eq4.1}) and  (\ref{Eq4.2}), we get the following
\begin{align} \label{Eq4.3}
&\
\sum_{j,\mu} g^i_{j,\mu} \sum_{k,\nu}  g^j_{k,\nu} (\varepsilon(z^{-1} x^j_{k,\nu})- \varepsilon((x^i_{j,\mu}+z)^{-1} x^j_{k,\nu})) \lambda_k \\ \notag
&\  = \sum_{k,\nu} g^i_{k,\nu} (\varepsilon(-z^{-1}(1+x^i_{k,\nu}))-\varepsilon(z^{-1}(z-1)^{-1}(x^i_{k,\nu}-z+1)) )\lambda_k
\end{align}
for any $z\in \bar{\mathbb{F}}^*_q$ such that  $x^i_{j,\mu}+z \ne 0$ and $i=1,2,\dots,n$. However it is not difficult to see that (\ref{Eq4.3}) is  not an identity. Indeed, for any fixed $i=1,2, \dots, n$, if there exists $x^i_{k_0,\nu_0}\ne -1$, then we can choose one $z_0 \in \bar{\mathbb{F}}^*_q$ such that $z_0^{-1}(z_0-1)^{-1}(x^i_{k_0,\nu_0}-z_0+1)) $ is different from the following set
$$\{z_0^{-1} x^j_{k,\nu},\ (x^i_{j,\mu}+z_0)^{-1} x^{j}_{k,\nu},\ -z_0^{-1}(1+x^i_{k,\nu}) \mid  g^j_{k,\nu} \ne0, g^i_{j,\mu}\ne 0, g^i_{k,\nu} \ne 0. \}.$$
If there is only one $x^i_{k,\nu}= -1$ such that $g^i_{k,\nu} \ne 0$, then (\ref{Eq4.3}) does not hold obviously.
Therefore $\mathcal{S}^n$ can not assign a $\Bbbk {\bf G}$-module structure. The lemma is proved.
\end{proof}

In general, using the same discussion as  in Lemma \ref{Notmodule}, we have the following proposition.

\begin{proposition}\normalfont \label{Notmodule2}
Let $\mathcal{S_{\lambda}}$ be the unique simple $\Bbbk {\bf B}$-submodule of $\text{Ind}_{\bf T}^{\bf B} \Bbbk_{\lambda}$  for each $\lambda\in \widehat{\bf T}$. One has that $\displaystyle \bigoplus_{\lambda \in \widehat{\bf T}}\mathcal{S_{\lambda}}^{n_{\lambda}}$ can not assign a $\Bbbk {\bf G}$-module structure, where $n_{\lambda}\in \mathbb{N}$ and only finitely many $n_{\lambda}$ are nonzero for $\lambda\in \widehat{\bf T}$.
\end{proposition}

\begin{theorem} \label{abeliancat}\normalfont
The category $\mathscr{C}(\bf G)$ is an abelian category.
\end{theorem}

\begin{proof} Recall the definition of $\mathscr{C}(\bf G)$ in Section 1, it is enough to show that $\mathscr{C}(\bf G)$ is closed under taking subquotients. Let $M$ be a $\Bbbk {\bf G}$-module with ${\bf T}$-eigenvectors $\xi_1, \xi_2, \dots, \xi_n$ such that $M\cong \displaystyle  \bigoplus_{i=1}^n \Bbbk {\bf U} \xi_i$ as $\Bbbk {\bf B}$-modules.
Following \cite[Theorem 2.5]{CD4}, when $\xi\in M_{\theta}$,  each $ \Bbbk {\bf U} \xi$ is isomorphic to $\text{Ind}_{\bf T}^{\bf B} \Bbbk_{\theta}$ or $\Bbbk_{\theta}$ . Thus $M$ is a $\Bbbk {\bf G}$-module of finite length. Let $N$ be a simple quotient of $M$, then $N$ has a  ${\bf T}$-eigenvector which implies $N\in \text{Irr}({\bf G})$. The Steinberg module $\text{St}$ is isomorphic to $\text{Ind}_{\bf T}^{\bf B} \Bbbk_{\text{tr}}$ as $\Bbbk {\bf B}$-modules and $ \mathbb{M}(\theta)$ is isomorphic to
$\Bbbk_{\theta} \oplus \text{Ind}_{\bf T}^{\bf B} \Bbbk_{\theta^s}$ as $\Bbbk {\bf B}$-modules. However since $\displaystyle \bigoplus_{\lambda \in \widehat{\bf T}}\mathcal{S_{\lambda}}^{n_{\lambda}}$ can not assign a $\Bbbk {\bf G}$-module structure by Proposition \ref{Notmodule2}, it is easy to see that the subquotients of $M$ are also in  $\mathscr{C}(\bf G)$. The theorem is proved.
\end{proof}

\begin{corollary} \label{finitelength}\normalfont
The category  $\mathscr{C}(\bf G)$ is noetherian and artinian.
\end{corollary}

A $\Bbbk {\bf G}$-module can be regarded as a $\Bbbk {\bf T}$-module (resp. $\Bbbk {\bf B}$-module) naturally. We denote by $\mathscr{C}(\bf G)_{\bf T}$ (resp. $\mathscr{C}(\bf G)_{\bf B}$)  the full subcategory of $\Bbbk {\bf T}$-Mod (resp. $\Bbbk {\bf B}$-Mod), which  consists of the objects in $\mathscr{C}(\bf G)$.

\begin{corollary} \normalfont
One has that
$\text{Hom}_{\bf T}(\Bbbk_{\theta}, -): \mathscr{C}(\bf G)_{\bf T}\rightarrow \text{Vect} $ is a exact functor for any $\theta\in \widehat{\bf T}$. Thus given a short exact sequence
$$0 \longrightarrow M_1  \longrightarrow M  \longrightarrow M_2  \longrightarrow 0$$
in the category $\mathscr{C}(\bf G)$, one has that $M^{\bf T} \cong M^{\bf T}_1 \oplus M^{\bf T}_2$ as ${\bf T}$-modules. In particular, we have $\text{Wt}(M)=\text{Wt}(M_1)\cup  \text{Wt}(M_2).$
\end{corollary}

\begin{proof} Using the setting  in \cite[Section 2]{FS}, when $\Bbbk$ is an algebraically closed field of characteristic zero, $\Bbbk {\bf B}$ is a locally Wedderburn algebra. Thus by \cite[Lemma 3]{FS}, $\Bbbk_{\theta}$ is an injective $\Bbbk {\bf B}$-module which implies the exactness of $\text{Hom}_{\bf B}(-, \Bbbk_{\theta})$. For $M\in \mathscr{C}(\bf G)$, we have $M\cong \displaystyle  \bigoplus_{i=1}^n \Bbbk {\bf U} \xi_i$ as $\Bbbk {\bf B}$-modules, where $\xi_i \in M_{\lambda_i}$. By \cite[Theorem 2.5]{CD4}, the induced module $\text{Ind}_{\bf T}^{\bf B} \Bbbk_{\theta}$ has a unique simple $\Bbbk {\bf B}$-submodule $\mathcal{S}_{\theta}$ and the corresponding quotient module is $\Bbbk_{\theta}$. Thus each $\Bbbk {\bf U} \xi_i$ is isomorphic to $\text{Ind}_{\bf T}^{\bf B} \Bbbk_{\lambda_i}$ or $\Bbbk_{\lambda_i}$ as $\Bbbk {\bf B}$-modules. Therefore it is easy to see that
$$ \text{Hom}_{\bf T} (\Bbbk_{\theta}, M) \cong \text{Hom}_{\bf B} (M, \Bbbk_{\theta}),$$
which implies the exactness of $\text{Hom}_{\bf T}(\Bbbk_{\theta}, -)$. The rest part is obvious.
\end{proof}

\begin{corollary} \label{Bsplit}\normalfont
Given a short exact sequence in $\mathscr{C}(\bf G)$,  then it is split  regarded as a short exact sequence in $\mathscr{C}(\bf G)_{\bf B}$.

\end{corollary}

\begin{proof} By the definition of $\mathscr{C}(\bf G)$ and Proposition \ref{Notmodule2}, noting that the set of simple objects in  the category $\mathscr{C}(\bf G)$ is
$$\text{Irr}({\bf G})= \{ \Bbbk_{\text{tr}}, \text{St}, \mathbb{M}(\theta)\mid \theta \in \widehat{{\bf T}} \ \text{is nontrivial}\},$$
the corollary is proved.
\end{proof}

\section{Highest weight category}

In this section, we show that $\mathscr{C}(\bf G)$ is a highest weight category for  ${\bf G} =SL_2(\bar{\mathbb{F}}_q)$. Firstly we recall the definition of highest weight category (see \cite{CPS}).

\begin{definition}\label{HWC} \normalfont

Let $\mathscr{C}$ be a locally artinian, abelian, $\Bbbk$-linear category with enough injective objects that satisfies Grothendieck's condition. Then we call $\mathscr{C}$ a highest weight category if there exists a locally finite poset $\Lambda$ (the ``weights" of $\mathscr{C}$), such that:

(a) There is a complete collection $\{S(\lambda)_{\lambda \in \Lambda}\}$ of non-isomorphic simple objects of $\mathscr{C}$ indexed by the set $\Lambda$.

(b) There is a collection $\{A(\lambda)_{\lambda \in \Lambda}\}$ of objects of $\mathscr{C}$ and, for each $\lambda$, an embedding $S(\lambda)\subset A(\lambda) $ such that all composition factors $S(\mu)$ of $A(\lambda)/S(\lambda)$ satisfy $\mu < \lambda$. For $\lambda,\mu \in \Lambda$,
we have that $\dim_{\Bbbk}\op{Hom}_{\mathscr{C}}(A(\lambda), A(\mu))$ and $[A(\lambda): S(\mu)]$ are finite.

(c) Each simple object $S(\lambda)$ has an injective envelope $\mathcal{I}(\lambda)$ in $\mathscr{C}$.
Also, $\mathcal{I}(\lambda)$  has a good filtration $0= F_0(\lambda)\subset F_1(\lambda)\subset \dots $ such that:

\noindent (i) $F_1(\lambda)\cong A(\lambda)$;

\noindent  (ii) for $n>1$, $F_n(\lambda)/F_{n-1}(\lambda) \cong A(\mu)$ for some $\mu=\mu(n)> \lambda$;

\noindent  (iii) for a given $\mu \in \Lambda$, $\mu=\mu(n)$ for only finitely many $n$;

\noindent  (iv) $\bigcup F_i(\lambda)= \mathcal{I}(\lambda)$.
\end{definition}

Actually, to show that $\mathscr{C}(\bf G)$ is a highest weight category, the main  difficulty is  to prove that the category $\mathscr{C}({\bf G})$ has enough injective objects.

\begin{proposition}  \label{zeroExt}\normalfont
 For any  $M, N$  in $\mathscr{C}(\bf G)$ such that $\text{Wt}(M) \cap \text{Wt}(N)=\emptyset$, we have $\text{Ext}^1_{\mathscr{C}(\bf G)}(M,N)=0$.
\end{proposition}

\begin{proof} Since each object in $\mathscr{C}(\bf G)$  is of finite length, it is enough to show that
$\text{Ext}^1_{\mathscr{C}(\bf G)}(M,N)=0$ for any simple object $M,N \in \mathscr{C}(\bf G)$, where $\text{Wt}(M) \cap \text{Wt}(N)=\emptyset$. Recall that $\text{Irr}({\bf G})= \{ \Bbbk_{\text{tr}}, \text{St}, \mathbb{M}(\theta)\mid \theta \in \widehat{{\bf T}} \ \text{is nontrivial}\}$ is the set of simple objects in $\mathscr{C}(\bf G)$,  we will show that
$$\text{Ext}^1_{\mathscr{C}(\bf G)}(\Bbbk_{\text{tr}},\mathbb{M}(\lambda))=0,\  \text{Ext}^1_{\mathscr{C}(\bf G)}(\text{St},\mathbb{M}(\mu))=0 \ \text{and}\ \text{Ext}^1_{\mathscr{C}(\bf G)}( \mathbb{M}(\theta),S)=0,$$
where $S$ is a simple object whose weights are different with $\theta$ and $\theta^s$.

\smallskip
\noindent (1) $\text{Ext}^1_{\mathscr{C}(\bf G)}(\Bbbk_{\text{tr}},\mathbb{M}(\lambda))=0$.
Suppose we have a short exact sequence
\begin{align} \label{SES1}
0 \longrightarrow \mathbb{M}(\lambda) \longrightarrow M  \longrightarrow \Bbbk_{\text{tr}}
\longrightarrow 0
\end{align}
in $\mathscr{C}(\bf G)$. Then using Corollary \ref{Bsplit}, there exists an element $\xi \in M $ such that $z\xi= \xi$ for any $z\in \bf B$. Noting that $M_{\text{tr}}=\Bbbk \xi$, thus $\dot{s}\xi=a\xi$ for some $a\in \Bbbk^*$. However the following equation
$$\dot{s} \varepsilon (x) \dot{s} \xi = \displaystyle \varepsilon(-x^{-1}) \dot{s} h(-x) \varepsilon(-x^{-1})\xi $$
shows that $a$ must be $1$. Therefore the short exact sequence (\ref{SES1}) is split which implies that $\text{Ext}^1_{\mathscr{C}(\bf G)}(\Bbbk_{\text{tr}},\mathbb{M}(\lambda))=0$.

\smallskip

\noindent (2)  $\text{Ext}^1_{\mathscr{C}(\bf G)}(\text{St},\mathbb{M}(\mu))=0$.
Suppose we have a short exact sequence
\begin{align} \label{SES2}
0 \longrightarrow \mathbb{M}(\mu) \longrightarrow N  \longrightarrow \text{St} \longrightarrow 0
\end{align}
in $\mathscr{C}(\bf G)$. By Corollary \ref{Bsplit}, there exists an element $\eta \in N_{\text{tr}}$ such that
$\dot{s}\eta =b\eta $ for some $b\in \Bbbk^*$ and  $\dot{s} \varepsilon(1) \eta=(\varepsilon (-1)-e ) \eta+ \varpi$ for some $\varpi \in \mathbb{M}(\mu)$.  Using $\dot{s}$ to act on both sides, we get $b=-1$. Moreover we have
$$\dot{s} \varepsilon (x) \eta=(\varepsilon (-x^{-1})-e ) \eta+ h(x^{-\frac{1}{2}})\varpi.$$
Now we consider the element $\dot{s}\varepsilon(z) \dot{s} \varepsilon (1) \eta$. Firstly we have
\begin{align} \label{Eq5.4}
\dot{s}\varepsilon(z) \dot{s} \varepsilon (1) \eta= \varepsilon(-z^{-1}) \dot{s} h(-z)\varepsilon(-z^{-1}) \varepsilon (1) \eta =\varepsilon(-z^{-1}) \dot{s} \varepsilon(z^2-z) \eta.
\end{align}
On the other hand, we get
\begin{align} \label{Eq5.5}
\dot{s}\varepsilon(z) \dot{s} \varepsilon(1) \eta=\dot{s} (\varepsilon (z-1)-\varepsilon(z) ) \eta+ \dot{s}\varepsilon(z)\varpi.
\end{align}
Compare the above two equations (\ref{Eq5.4}) and  (\ref{Eq5.5}), it is not difficult to  get
\begin{align} \label{Eq5.3}
\dot{s}\varepsilon(z)\varpi=\varepsilon(-z^{-1})h((z^2-z)^{-\frac{1}{2}}))\varpi+ ( h(z^{-\frac{1}{2}})-h((z-1)^{-\frac{1}{2}}))\varpi
\end{align}
for any $z\in \bar{\mathbb{F}}_q$.
Denote by
$$\varpi= \displaystyle f {\bf 1}_{\mu}+ \sum_{x\in \bar{\mathbb{F}}_q}g_x \epsilon(x)\dot{s} {\bf 1}_{\mu}, \quad \text{where} \ f,g_x \in \bar{\mathbb{F}}_q.$$ Then it is easy to see that $f=0$.
Moreover if $g_x\ne 0$ when $x\ne 0, -1$, then (\ref{Eq5.3}) can not hold for any $z\in \bar{\mathbb{F}}_q$.
Substitute $\varpi=a\dot{s}{\bf 1}_{\mu} +b\varepsilon(-1)\dot{s}{\bf 1}_{\mu}  $ into (\ref{Eq5.3}) and we get $\varpi=0$. Thus the short exact sequence (\ref{SES2}) is split which implies that $\text{Ext}^1_{\mathscr{C}(\bf G)}(\text{St},\mathbb{M}(\mu))=0$.

\smallskip

\noindent (3) $\text{Ext}^1_{\mathscr{C}(\bf G)}( \mathbb{M}(\theta),S)=0$ for some  simple object $S$ whose weights are different with $\theta$ and $\theta^s$. Suppose we have a short exact sequence
\begin{align} \label{SES3}
0 \longrightarrow S \longrightarrow L \longrightarrow \mathbb{M}(\theta) \longrightarrow 0
\end{align}
in $\mathscr{C}(\bf G)$. Then by Corollary \ref{Bsplit}, there exists $\zeta\in L_{\theta}$ such that $u \zeta =\zeta$  and thus  $\dot{s}\zeta\in L_{\theta^s}$ since the weights of $S$ are different with $\theta$ and $\theta^s$. Thus the short exact sequence (\ref{SES3}) is split and  we have $\text{Ext}^1_{\mathscr{C}(\bf G)}( \mathbb{M}(\theta),S)=0$.

\end{proof}

\begin{proposition} \label{injectivemodule}\normalfont
One has that $\Bbbk_{\text{tr}}$ and  $\mathbb{M}(\theta)$ are injective objects in the category $\mathscr{C}(\bf G)$ for any $\theta \in \widehat{\bf T}$.
\end{proposition}

\begin{proof} By \cite[Theorem 1]{FS}, we see that $\Bbbk_{\text{tr}}$ is an injective $\Bbbk {\bf G}$-module and thus it is injective in  $\mathscr{C}(\bf G)$. According to Proposition \ref{zeroExt}, it is enough to show that
$$\text{Ext}^1_{\mathscr{C}(\bf G)}(\Bbbk_{\text{tr}},\mathbb{M}(\text{tr}))=0,\  \text{Ext}^1_{\mathscr{C}(\bf G)}(\text{St},\mathbb{M}(\text{tr})))=0,\ \text{Ext}^1_{\mathscr{C}(\bf G)}( \mathbb{M}(\lambda),\mathbb{M}(\theta))=0,$$
where $\lambda= \theta$ or $\theta^s$.

\smallskip

\noindent (1) $\text{Ext}^1_{\mathscr{C}(\bf G)}(\Bbbk_{\text{tr}},\mathbb{M}(\text{tr}))=0$.
Since there is a short exact sequence
$$0 \longrightarrow \text{St} \longrightarrow \mathbb{M}(\text{tr}) \longrightarrow \Bbbk_{\text{tr}}\longrightarrow 0, $$
we get a long exact sequence
$$
\aligned  &\  0 \longrightarrow   \text{Hom}_{\mathscr{C}(\bf G)}(\Bbbk_{\text{tr}},\text{St}) \longrightarrow \text{Hom}_{\mathscr{C}(\bf G)}(\Bbbk_{\text{tr}},\mathbb{M}(\text{tr}))
 \longrightarrow \text{Hom}_{\mathscr{C}(\bf G)}(\Bbbk_{\text{tr}}, \Bbbk_{\text{tr}}) \\
&\  \longrightarrow \text{Ext}^1_{\mathscr{C}(\bf G)}(\Bbbk_{\text{tr}},\text{St})
\longrightarrow \text{Ext}^1_{\mathscr{C}(\bf G)}(\Bbbk_{\text{tr}},\mathbb{M}(\text{tr}))
 \longrightarrow \text{Ext}^1_{\mathscr{C}(\bf G)}(\Bbbk_{\text{tr}}, \Bbbk_{\text{tr}})\longrightarrow \dots.
\endaligned
$$
It is easy to check that $\text{Hom}_{\mathscr{C}(\bf G)}(\Bbbk_{\text{tr}},\mathbb{M}(\text{tr}))=0$ and
$\text{Hom}_{\mathscr{C}(\bf G)}(\Bbbk_{\text{tr}}, \Bbbk_{\text{tr}})\cong \Bbbk$. Noting that
$\text{Ext}^1_{\mathscr{C}(\bf G)}(\Bbbk_{\text{tr}}, \Bbbk_{\text{tr}})=0$ because  $\Bbbk_{\text{tr}}$ is injective in  $\mathscr{C}(\bf G)$, it is enough to verify that $\text{Ext}^1_{\mathscr{C}(\bf G)}(\Bbbk_{\text{tr}},\text{St})\cong  \Bbbk$. Given a short exact sequence
$$0 \longrightarrow \text{St} \longrightarrow M \longrightarrow \Bbbk_{\text{tr}}\longrightarrow 0$$
in  $\mathscr{C}(\bf G)$,  then by Corollary \ref{Bsplit}, we get
$$ \text{Hom}_{{\bf G}}(\mathbb{M}(\text{tr}), M)\cong \text{Hom}_{{\bf B}}(\Bbbk_{\text{tr}}, M)\cong \Bbbk$$ using Frobenius reciprocity. Thus $M\cong \mathbb{M}(\text{tr})$ or $M\cong \text{St} \oplus \Bbbk_{\text{tr}}$. Therefore $\text{Ext}^1_{\mathscr{C}(\bf G)}(\Bbbk_{\text{tr}},\text{St})\cong  \Bbbk$ and hence $\text{Ext}^1_{\mathscr{C}(\bf G)}(\Bbbk_{\text{tr}},\mathbb{M}(\text{tr}))=0$.

\smallskip

\noindent (2) $\text{Ext}^1_{\mathscr{C}(\bf G)}(\text{St},\mathbb{M}(\text{tr}))=0$.
Using the short exact sequence
$$0 \longrightarrow \text{St} \longrightarrow \mathbb{M}(\text{tr}) \longrightarrow \Bbbk_{\text{tr}}\longrightarrow 0, $$
we get a long exact sequence
$$
\aligned  &\  0 \longrightarrow   \text{Hom}_{\mathscr{C}(\bf G)}(\text{St},\text{St}) \longrightarrow \text{Hom}_{\mathscr{C}(\bf G)}(\text{St},\mathbb{M}(\text{tr}))
 \longrightarrow \text{Hom}_{\mathscr{C}(\bf G)}(\text{St}, \Bbbk_{\text{tr}}) \\
&\  \longrightarrow \text{Ext}^1_{\mathscr{C}(\bf G)}(\text{St},\text{St})
\longrightarrow \text{Ext}^1_{\mathscr{C}(\bf G)}(\text{St},\mathbb{M}(\text{tr}))
 \longrightarrow \text{Ext}^1_{\mathscr{C}(\bf G)}(\text{St}, \Bbbk_{\text{tr}})\longrightarrow \dots.
\endaligned
$$
Since  $\Bbbk_{\text{tr}}$ is injective and we have $\text{Ext}^1_{\mathscr{C}(\bf G)}(\text{St}, \Bbbk_{\text{tr}})=0$, it is enough to show that $\text{Ext}^1_{\mathscr{C}(\bf G)}(\text{St},\text{St})=0$.
Given a short exact sequence
\begin{align} \label{SES4}
0 \longrightarrow \text{St} \longrightarrow N \longrightarrow \text{St}\longrightarrow 0
\end{align}
in  $\mathscr{C}(\bf G)$, then there exists $\xi_1,\xi_2 \in N_{\text{tr}}$ such that $N\cong \Bbbk {\bf U}\xi_1 \oplus \Bbbk {\bf U} \xi_2 $ as $\Bbbk {\bf B}$-modules and $\Bbbk {\bf U}\xi_1$ is isomorphic to the Steinberg module.
Firstly it is easy to see that  $\dot{s} \xi_2= -\xi_2 +a\xi_1$ for some $a\in \Bbbk$. Using $\dot{s}$ to act on both sides, we get $a=0$ and thus $\dot{s} \xi_2= -\xi_2$. On the other hand, we have
$$\dot{s} \varepsilon (1) \xi_2=(\varepsilon (-1)-e ) \xi_2+ \varpi$$ for some $\varpi \in \Bbbk {\bf U} \xi_1$. Denote by $ \varpi=\displaystyle  \sum_{x\in \bar{\mathbb{F}}_q}a_x \varepsilon(x) \xi_1$ and then we get
$$\dot{s} \varepsilon (y) \xi_2=(\varepsilon (-y^{-1})-e ) \xi_2+ \sum_{x\in \bar{\mathbb{F}}_q}a_x \varepsilon(xy^{-1}) \xi_1$$
for any $y\in \bar{\mathbb{F}}^*_q$. Now we consider the element $\dot{s}\varepsilon (z)
\dot{s} \varepsilon(1) \xi_2$. Firstly we have
\begin{align}\label{Eq5.1}
\dot{s}\varepsilon (z)\dot{s} \varepsilon (1) \xi_2=\varepsilon(-z^{-1}) \dot{s} h(-z)\varepsilon(-z^{-1}) \varepsilon (1) \xi_2 =\varepsilon(-z^{-1}) \dot{s} \varepsilon(z^2-z) \xi_2.
\end{align}
On the other hand, we have
\begin{align}\label{Eq5.2}
\dot{s}\varepsilon(z) \dot{s} \varepsilon(1) \xi_2=\dot{s} (\varepsilon (z-1)-\varepsilon(z) ) \xi_2+ \dot{s}\varepsilon(z)\varpi.
\end{align}
Combining (\ref{Eq5.1}) and (\ref{Eq5.2}),  it is not difficult to get
\begin{align} \label{Eq5.6}
\dot{s}\varepsilon(z) \sum_{x\in \bar{\mathbb{F}}_q}a_x \varepsilon(x) \xi_1 = \sum_{x\in \bar{\mathbb{F}}_q}a_x( \varepsilon(x(z^2-z)^{-1}-z^{-1})+\varepsilon(xz^{-1})-\varepsilon(x(z-1)^{-1})   )\xi_1
\end{align}
for any $z\in \bar{\mathbb{F}}_q$. If $a_x\ne 0$ when $x\ne 0, -1$, then the equation can not hold for any $z\in \bar{\mathbb{F}}_q$. However, substitute $\varpi=a\xi_1 +b\varepsilon(-1)\xi_1 $ into (\ref{Eq5.6}), we get $\varpi=0$ easily. Thus the short exact sequence (\ref{SES4}) is split and we have $\text{Ext}^1_{\mathscr{C}(\bf G)}(\text{St},\text{St})=0$ which implies that $\text{Ext}^1_{\mathscr{C}(\bf G)}(\text{St},\mathbb{M}(\text{tr}))=0$.

\smallskip
\noindent (3) $\text{Ext}^1_{\mathscr{C}(\bf G)}( \mathbb{M}(\lambda),\mathbb{M}(\theta))=0$
for $\lambda= \theta$ or $\theta^s$. Firstly given a short exact sequence
\begin{align} \label{SES5}
0 \longrightarrow \mathbb{M}(\theta)  \longrightarrow K \longrightarrow \mathbb{M}(\theta) \longrightarrow 0
\end{align}
in  $\mathscr{C}(\bf G)$, then by Corollary \ref{Bsplit}, we have
$$K \cong \Bbbk_{\theta} \oplus \Bbbk_{\theta} \oplus\text{Ind}_{\bf T}^{\bf B} \Bbbk_{\theta^s} \oplus\text{Ind}_{\bf T}^{\bf B} \Bbbk_{\theta^s} $$
as $\Bbbk {\bf B}$-modules and there exists $\xi_1, \eta_1 \in K_{\theta}$ and $\xi_2, \eta_2\in K_{\theta^s}$ such that $ \Bbbk \xi_1+\Bbbk {\bf U}\xi_2 \cong \mathbb{M}(\theta)$, which is a $\Bbbk {\bf G}$-submodule of $K$. We can assume that $\xi_2=\dot{s}\xi_1$. Now suppose
$$\dot{s} \eta_1 =a\xi_2 +b\eta_2, \quad \dot{s} \eta_2 =c\xi_1+ d\eta_1$$
for some $a,b,c,d \in \Bbbk$. Using $\dot{s}$ to act on both sides of the two equations, we get $c=\displaystyle -\frac{a}{b}$ and $d=\displaystyle \frac{1}{b}$. In particular, $b$ is nonzero.
Now set $\widetilde{\eta_2}= a\xi_2 +b\eta_2$. Thus $\Bbbk \eta_1+\Bbbk {\bf U}\widetilde{\eta_2}$ is isomorphic to $\mathbb{M}(\theta)$. So the short exact sequence (\ref{SES5}) is split which implies that $\text{Ext}^1_{\mathscr{C}(\bf G)}( \mathbb{M}(\theta),\mathbb{M}(\theta))=0$.

Next we consider the short exact sequence
\begin{align} \label{SES6}
0 \longrightarrow \mathbb{M}(\theta)  \longrightarrow L \longrightarrow \mathbb{M}(\theta^s) \longrightarrow 0
\end{align}
in  $\mathscr{C}(\bf G)$.
There exists $\zeta_1, \varrho_1 \in L_{\theta}$ and $\zeta_2, \varrho_2 \in L_{\theta^s}$ such that
$\mathbb{M}(\theta) \cong \Bbbk \zeta_1+\Bbbk {\bf U}\zeta_2$.
We can assume that $\zeta_2=\dot{s}\zeta_1$. Now suppose
$$\dot{s} \varrho_1 =a'\zeta_2 +b'\varrho_2, \quad \dot{s} \varrho_2 =c'\zeta_1+ d'\varrho_1,$$
then we also get $c'=\displaystyle -\frac{a'}{b'}$ and $d'=\displaystyle \frac{1}{b'}$.
In particular, $b'$ and $d'$ are nonzero. Set $\widetilde{\varrho_1 }= c'\zeta_1+ d'\varrho_1$, then we get $\mathbb{M}(\theta^s) \cong \Bbbk \varrho_2+\Bbbk {\bf U}\widetilde{\varrho_1 }$. Therefore the short exact sequence (\ref{SES6}) is split and we also get $\text{Ext}^1_{\mathscr{C}(\bf G)}(\mathbb{M}(\theta^s),\mathbb{M}(\theta))=0$. The proposition is proved.

\end{proof}

\begin{theorem} \label{enough injectives} \normalfont
The category $\mathscr{C}({\bf G})$ has enough injective objects.
\end{theorem}
\begin{proof} By Proposition \ref{injectivemodule}, the injective envelope for each simple object  exists in  the category $\mathscr{C}({\bf G})$. Thus a standard argument shows that  $\mathscr{C}({\bf G})$ has enough injective objects (see \cite[Theorem 2.7]{D1}).

\end{proof}

Now we show that $\mathscr{C}({\bf G})$ is a highest weight category. In Definition \ref{HWC},
the set of  weights  is $\Lambda= \{(\theta, J )\ | \ \theta \in \widehat{\bf T}, J \subset I(\theta) \}$ and we define the order of the weights by
$$(\theta_1, J_1) \leq (\theta_2, J_2 ), \ \text{if}\ \theta_1=\theta_2 \ \text{and} \ J_1\supseteq J_2.$$
Specifically, for ${\bf G} =SL_2(\bar{\mathbb{F}}_q)$, $\Lambda= \{(\text{tr}, \emptyset ),(\text{tr}, \{s\}), (\theta, \emptyset)\ | \ \theta \  \text{is nontrivial} \}$.
Let $S(\text{tr},  \{s\} )=A(\text{tr}, \{s\})=\Bbbk_{\text{tr}}$, $S(\text{tr}, \emptyset )=A(\text{tr}, \emptyset )=\text{St}$ and $S(\theta, \emptyset)=A(\theta, \emptyset)=\mathbb{M}(\theta)$.
 By Proposition \ref{injectivemodule}, we set $\mathcal{I}(\text{tr}, \{s\})= \Bbbk_{\text{tr}}$, $\mathcal{I}(\text{tr}, \emptyset)= \mathbb{M}(\text{tr})$ and $\mathcal{I}(\theta, \emptyset)=\mathbb{M}(\theta)$.
 It is not difficult to check that $\mathscr{C}({\bf G})$ with this setting satisfies  all the condition in Definition \ref{HWC} and thus we have the following theorem.

\begin{theorem} \label{Highestweicat} \normalfont
The category $\mathscr{C}({\bf G})$  is a highest weight category.
\end{theorem}

According to the same discussion as \cite[Section 4 and Section 5]{D1}, all the indecomposable modules in $\mathscr{C}({\bf G})$ are $\op{St}, \Bbbk_{\op{tr}}$ and $\{\mathbb{M}(\theta)\mid \theta \in \widehat{\bf T} \}$. Therefore each  object in $\mathscr{C}({\bf G})$ is a direct sum of these modules. In particular,
the category $\mathscr{C}({\bf G})$ is a Krull–Schmidt category.

\begin{remark} \normalfont
In \cite{Chen}, X.Y. Chen constructed a complex representation  $M$ of  $SL_2(\bar{\mathbb{F}}_q)$, which contains the Steinberg module as a proper submodule and the corresponding quotient module is the trivial  module. However $M$ has no ${\bf B}$-stable line. So, this gives a negative answer to \cite[Conjecture 3.7]{D1}. Thus the principal representation category $\mathscr{O}({\bf G})$ introduced in \cite{D1} is not a highest weight category. Chen's work shows that $\mathscr{O}({\bf G})$ is very complicated in general. However the representation $M$ he constructed in \cite{Chen} is not in the category $\mathscr{C}(SL_2(\bar{\mathbb{F}}_q))$. Therefore, for general reductive algebraic group $\bf G$, the category $\mathscr{C}({\bf G})$ may be a good category for further study.
\end{remark}

\medskip

\noindent{\bf Acknowledgements.} The author is grateful to Nanhua Xi, Ming Fang and  Xiaoyu Chen for their suggestions and helpful discussions. The work is sponsored by Shanghai Sailing Program (No.21YF1429000) and NSFC-12101405.

\bigskip

\bibliographystyle{amsplain}

\end{document}